\documentclass[11pt]{amsart}

\usepackage{amsrefs,amstext,amsmath,amsthm,amsfonts,amssymb,enumerate,amscd,cite,pb-diagram}

\newtheorem{theorem}{Theorem}[section]
\newtheorem{proposition}[theorem]{Proposition}
\newtheorem{lemma}[theorem]{Lemma}
\newtheorem{corollary}[theorem]{Corollary}

\newtheorem{examples}[theorem]{Examples}
\newtheorem{definition}[theorem]{Definition}
\newtheorem{remark}[theorem]{Remark}
\newtheorem{remarks}[theorem]{Remarks}

\newtheorem{problem}[theorem]{Problem}

\newcommand{\CC}{\ensuremath{\mathbb{C}}} 
\newcommand{\RR}{\ensuremath{\mathbb{R}}} 

\newcommand{\id}{\mathrm{id}} 
\newcommand{\Ad}{\mathrm{Ad}} 
\newcommand{\defeq}{:=} 
\newcommand{\inv}{^{-1}} 
\newcommand{\centre}[1][M]{\mathcal{Z}({#1})} 
\newcommand{\unitiz}[1][A]{{#1}_{(1)}} 

\newcommand{\HH}{\ensuremath{\mathcal{H}}} 
\newcommand{\ip}[2]{\left\langle #1 , #2 \right\rangle} 

\newcommand{\nrm}[1][\cdot]{\Vert #1 \Vert} 
\newcommand{\dualp}[2]{\left\langle #1 , #2 \right\rangle} 

\newcommand{\Bd}[1][\HH]{\ensuremath{\mathcal{B}(#1)}} 
\newcommand{\Cpt}[1][\HH]{\ensuremath{\mathcal{K}(#1)}} 
\newcommand{\btens}{\mathbin{\overline{\otimes}}} 

\newcommand{\grp}{G} 
\newcommand{\Aut}{\mathop{\mathrm{Aut}}} 
\newcommand{\lreg}{\lambda} 
\newcommand{\SL}{\operatorname{SL}} 

\newcommand{\vN}[1][\grp]{\ensuremath{\mathrm{vN}(#1)}} 
\newcommand{\rgCst}[1][G]{\ensuremath{C^*_\lreg(#1)}} 
\newcommand{\cros}[3]{\ensuremath{#1 \rtimes_{#3} \! #2}} 
\newcommand{\rcros}[3]{\ensuremath{#1 \rtimes_{{#3},r}  #2}} 
\newcommand{\crs}{\cros{A}{G}{\alpha}} 
\newcommand{\rcrs}{\rcros{A}{G}{\alpha}} 
\newcommand{\vNcros}[3]{\ensuremath{#1 \rtimes_{#3} \! #2}} 
\newcommand{\vNcrs}{\vNcros{M}{\grp}{\alpha}} 

\newcommand{\piact}[1][\alpha]{\pi_{#1}} 
\newcommand{\lrega}{\lambda} 
\newcommand{\coact}[1][\rgCst]{\Delta_{#1}} 

\newcommand{\falg}[1][\grp]{\ensuremath{A(#1)}} 
\newcommand{\unactC}[2]{{#1}^{**}_{#2}} 
\newcommand{\Aact}{\unactC{A}{\alpha}}

\newcommand{\CB}[1][A]{\ensuremath{\mathcal{CB}(#1)}} 
\newcommand{\cb}{{\mathrm{cb}}} 
\newcommand{\projtens}{\mathbin{\hat{\otimes}}} 

\newcommand{\Mult}{\mathrm{M}} 

\newcommand{\modu}{\mathbf{mod}}

\title[Amenable and inner amenable actions]{Amenable and inner amenable actions and approximation properties for crossed products by locally compact groups}

\author{Andrew McKee}
\address{Faculty of Mathematics, University of Bia\l ystok, ul.\ K.\ Cio\l kowskiego 1M, 15-425 Bia\l ystok, Poland}
\email{amckee240@qub.ac.uk}

\author{Reyhaneh Pourshahami}
\address{Department of Mathematics, Kharazmi University, 50 Taleghani Ave., 15618, Tehran, Iran}
\email{std\_reyhaneh.pourshahami@khu.ac.ir}

\begin{document}

\maketitle

\begin{abstract}
Amenable actions of locally compact groups on von~Neumann algebras are investigated by exploiting the natural module structure of the crossed product over the Fourier algebra of the acting group. 
The resulting characterisation of injectivity for crossed products generalises a result of Anantharaman-Delaroche on discrete groups.
Amenable actions of locally compact groups on $C^*$-algebras are investigated in the same way, and amenability of the action is related to nuclearity of the corresponding crossed product. A survey is given to show that this notion of amenable action for $C^*$-algebras satisfies a number of expected properties.
A notion of inner amenability for actions of locally compact groups is introduced, and a number of applications are given in the form of averaging arguments, relating approximation properties of crossed product von~Neumann algebras to properties of the components of the underlying $w^*$-dynamical system. 
We use these results to answer a recent question of Buss--Echterhoff--Willett.
\end{abstract}

\section{Introduction}
\label{sec:intro}

In the setting of group actions on operator algebras there are a number of questions which have satisfactory answers for actions of discrete groups, but which are still open for the case of actions of locally compact groups.
\begin{enumerate}
	\item 
		\begin{enumerate}
			\item How is injectivity of the crossed product corresponding to a $w^*$-dynamical system $(M,\grp,\alpha)$ related to amenability of the action $\alpha$ and injectivity of $M$? (See \cite[Corollaire 4.2]{AD79} for the answer in the discrete case.)
			\item How is nuclearity of the (full or reduced) crossed product of a $C^*$-dynamical system $(A,G,\alpha)$ related to amenability of the action $\alpha$ and nuclearity of $A$? (See \cite[Th\'{e}or\`{e}me 4.5]{AD87} for the answer in the discrete case.)
		\end{enumerate}
	\item How can one carry out Haagerup's averaging arguments, in order to pass from approximation properties of a crossed product to approximation properties of the components of a ($C^*$ or $w^*$)-dynamical system? (See \textit{e.g.}\ \cite[Proposition 3.4]{MSTT18} for the answer in the discrete case.)
\end{enumerate}
In this paper we show that these problems can be solved by accounting for the module structure over the Fourier algebra of the acting group, which crossed products naturally carry. 

Restricting for a moment to group $C^*$-algebras and group von~Neumann algebras (the case of trivial actions) question (1) goes back to Lance~\cite[Theorem 4.3]{Lan73}, who showed that a discrete group $G$ is amenable  if and only if $\rgCst$ is a nuclear $C^*$-algebra (and a similar technique shows $\vN$ is an injective von~Neumann algebra). 
Connes~\cite{Con76} showed that this does not generalise directly to locally compact groups; indeed, though amenable locally compact groups have injective group von~Neumann algebras, there exist nonamenable locally compact groups $G$ for which $\vN$ is injective (any nonamenable connected Lie group is an example).
Lau and Paterson~\cite[page 85]{Pat88} recognised that it is only in the class of inner amenable groups that injectivity of $\vN$ implies amenability of $G$. Since every amenable group is inner amenable this provides a complete answer to the problem of understanding amenability of locally compact groups in terms of their group von~Neumann algebras.
Recently Crann and Tanko~\cite[Theorem 3.4]{CrTa17} gave a reinterpretation of this result in terms of module maps: it is injectivity of $\vN$ \emph{as an $\falg$-module} which implies that $G$ is amenable, and inner amenable groups are the locally compact groups for which injectivity of $\vN$ implies injectivity of $\vN$ as an $\falg$-module --- see Theorem~\ref{th:laupat}.
All discrete groups are inner amenable, so in the discrete case one does not need to account for the $\falg$-module structure.

For trivial actions the averaging arguments mentioned in (2) were introduced by Haagerup~\cite{H16} (written in 1986 and circulated as an unpublished manuscript), where such arguments were used to study discrete groups. 
Later work of Lau--Paterson~\cite[page 85]{Pat88} indicated that one can use inner amenability, rather than discreteness, to give such averaging arguments, extending these techniques beyond discrete groups. 
Crann~\cite{Cra19} (see also \cite[Section 3]{CrTa17}) rephrases inner amenability in terms of relative module injectivity, clarifying the exact relationship between amenability of a locally compact group and injectivity of its group von~Neumann algebra.

Moving on to group actions, Anantharaman-Delaroche~\cite{AD79} introduced a definition of amenable action of a locally compact group on a von~Neumann algebra, building on Zimmer's work \cites{Zim77a,Zim77b,Zim78} on amenable actions on commutative von~Neumann algebras. 
For actions of discrete groups this definition is used to solve question (1) part (a) \cite[Corollaire 4.2]{AD79}, and was subsequently adapted to solve question (1) part (b) \cite[Th\'{e}or\`{e}me 4.5]{AD87}.
We caution the reader that there is another definition of amenable action on a set in the literature which is different from Zimmer's; in this paper we consider only Zimmer's notion and its generalisation by Anantharaman-Delaroche. 
The question of how to define amenable actions of locally compact groups on $C^*$-algebras remained open (see \cite[Section 9]{AD02}). 
This question has attracted significant recent attention --- we note work of Bearden--Crann~\cite{BeCr20a}, Buss--Echterhoff--Willett~\cites{BEW20a,BEW20b}, Ozawa--Suzuki~\cite{OzSu20} and  Suzuki~\cite{Suz19}, all of which are concerned with amenable actions of locally compact groups on $C^*$-algebras.

Since we began this work there has been a flurry of papers on the topic of amenable actions of locally compact groups.
As well as the papers \cites{BeCr20a,BEW20a,BEW20b,OzSu20,Suz19} mentioned above, which focus on amenable actions on $C^*$-algebras, we have recent work of Crann~\cite{Cra20} and Bearden--Crann~\cite{BeCr20b}.
We have rewritten this paper several times in attempts to account for these works.
In particular, Crann and Bearden--Crann \cites{Cra20,BeCr20b} also discovered the relationship between amenable actions and module injectivity, and gave another related module property. 
Our results were obtained independently, but we note that we have been heavily influenced by Crann's earlier work on module injectivity for quantum groups \cites{Cra17,Cra19}.

The organisation of this paper is as follows. 
In Section~\ref{sec:prelims} we review the definitions and results surronding the notion of group operator algebras and crossed products. Then we review the module versions of injectivity, and their link to amenability, which we aim to generalise. 
In Section~\ref{sec:amenablevNcase} we generalise amenable actions on von~Neumann algebras by using the natural module structure of a crossed product over the Fourier algebra. 
In Section~\ref{sec:inneramenableactions} we introduce inner amenable actions on von~Neumann algebras. We give some basic properties of this notion and extend Lau--Paterson's result \cite[Theorem 3.1]{LaPa91} to inner amenable actions on non-commutative injective von~Neumann algebras.
In Section~\ref{sec:approximation} we consider the weak* completely bounded approximation property for crossed product von~Neumann algebras, giving an example of how inner amenable actions can be used for averaging arguments in Proposition~\ref{pr:amenactCBAP}. 
Section~\ref{sec:amenactCalgs} is devoted to amenable actions on $C^*$-algebras. 
We use the natural extension of Anantharaman-Delaroche's definition of amenable action of a discrete group on a $C^*$-algebra to locally compact groups, and relate this to nuclearity of the corresponding crossed products.
We also give a survey of known results, mainly from work of Buss--Echterhoff--Wilett~\cite{BEW20a}, Bearden--Crann~\cite{BeCr20a} and Suzuki~\cite{Suz19}, to demonstrate how close our notion of amenable action comes to satisfying the requirements for an amenable action on a $C^*$-algebra suggested by Anantharaman-Delaroche~\cite[Section 9.2]{AD02}.

\section{Preliminaries}
\label{sec:prelims}

For von~Neumann algebras $M$ and $N$ we will write their normal spatial tensor product as $M \btens N$. 
We will use throughout the theory of operator spaces and completely bounded maps; the reference \cite{EfRu00} is suggested for this background material.

\subsection{Group operator algebras and crossed products}
\label{ssec:multiplierprelims}

Let $\grp$ be a locally compact group.
We write $\lreg$ for the \emph{left regular representation}
\[
    \lreg : \grp \to \Bd[L^2(\grp)] ;\ \lreg_r \xi(s) \defeq \xi(r\inv s) , \quad r,s \in \grp , \ \xi \in L^2(\grp) ,
\]
which is a unitary representation of $\grp$, so extends to a $*$-representation of $C_c(\grp)$ on $L^2(\grp)$ by
\[
    \lreg : C_c(\grp) \to \Bd[L^2(\grp)] ;\ \lreg(f) \defeq \int_\grp f(r) \lreg_r dr , \quad f \in C_c(\grp) .
\]
The \emph{reduced group $C^*$-algebra} $\rgCst$ is the $C^*$-algebra obtained by completing $C_c(\grp)$ in the norm induced by $\lreg$, while the group von~Neumann algebra is the double commutant $\vN \defeq \rgCst'' = \{ \lreg_r : r \in G \}''$. 

The \emph{Fourier algebra of $\grp$}, denoted by $\falg$, is the set
\[
    \falg = \{ u : \grp \to \CC : u(r) = \ip{\lreg_r \xi}{\eta} , \ \xi,\eta \in L^2(\grp) \} ,
\]
with pointwise operations and the norm $\nrm[u] = \inf \nrm[\xi] \nrm[\eta]$, where the infimum is taken over all possible representations $u(\cdot) = \ip{\lreg_{\cdot} \xi}{\eta}$.
The Fourier algebra is a completely contractive Banach algebra, and is naturally identified with the predual of $\vN$ by the pairing $\dualp{\lreg_r}{u} = u(r)$ ($u \in \falg$).
We refer to \cite{Eym64} for further background. 

Now let $A$ be a $C^*$-algebra, and $\alpha : \grp \to \Aut(A)$ a homomorphism which is continuous in the point-norm topology, \textit{i.e.}\ let $(A , \grp , \alpha)$ be a $C^*$-dynamical system.
A pair $(\phi , \rho)$, where $\phi : A \to \Bd$ is a $*$-representation and $\rho : \grp \to \Bd$ is a unitary representation, is called a \emph{covariant pair for $(A,G,\alpha)$} if $\rho_r \phi(a) \rho_r^* = \phi(\alpha_r(a))$ for all $r \in \grp,\ a \in A$.
In particular, a \emph{regular covariant pair} is constructed as follows: suppose that $\pi : A \to \Bd$ is a faithful $*$-representation and define
\begin{equation}\label{eq:covpair}
\begin{gathered}
    \piact : A \to \Bd[L^2(\grp) \otimes \HH] ; \ \piact(a) \xi(s) \defeq \alpha_{s\inv}(a) \xi(s) , \\ \lrega : \grp \to \Bd[L^2(\grp) \otimes \HH] ; \ \lrega_r \xi(s) \defeq \xi(r\inv s) .
\end{gathered}
\end{equation}
A covariant pair $(\phi , \rho)$ as above defines a $*$-representation of $C_c(\grp , A)$ on $\HH$ by
\[
    \phi \rtimes \rho (f) \defeq \int_\grp \phi \big( f(r) \big) \rho_r dr , \quad f \in C_c(\grp , A) .
\]
The \emph{reduced crossed product} $\rcrs$ is the $C^*$-algebra obtained by completing $C_c(\grp ,A)$ in the norm induced by $\piact \rtimes \lrega$; it can be shown that this definition is independent of the original faithful $*$-representation $\pi$. 
The \emph{full crossed product} $\crs$ is the completion of $C_c(G,A)$ in the universal norm
\[
	\nrm[f] \defeq \sup \{ \nrm[\phi \rtimes \rho(f)] : \text{$(\phi,\rho)$ is a covariant pair for $(A,G,\alpha)$} \} .
\]
We refer to Williams~\cite[Chapter 2]{Wil07} for the details of this construction, and the fact that the full crossed product is universal in the following sense: 
every $*$-representation of the full crossed product $\crs$ is of the form $\phi \rtimes \rho$ for some covariant pair $(\phi , \rho)$ for $(A,\grp,\alpha)$.
In particular we will use the covariant pair $(i_A , i_\grp)$ for which $i_A \rtimes i_\grp$ is the universal $*$-representation of $\crs$ in Section~\ref{sec:amenactCalgs}.


If $M \subset \Bd$ is a von~Neumann algebra and $\alpha$ is a point-weak* continuous action of $\grp$ on $M$ then we say that $(M,\grp,\alpha)$ is a \emph{$w^*$-dynamical system}. 
In this case the definition of $\piact$ in (\ref{eq:covpair}) is an injective, weak*-continuous homomorphism $\piact : M \to L^\infty(\grp) \btens M$ satisfying 
\[
	(\id \otimes \piact) \circ \piact = (\coact[L^\infty(\grp)] \otimes \id) \circ \piact .
\]
Here $\coact[L^\infty(\grp)]$ is the natural coproduct on $L^\infty(\grp)$, given by $\coact[L^\infty(\grp)](\phi)(s,t) \defeq \phi(st)$ ($\phi \in L^\infty(\grp),\ s,t \in \grp$).
The crossed product associated to $(M,\grp,\alpha)$ is the von~Neumann algebra $\vNcrs \defeq \{ \pi_\alpha(M) , \vN \otimes \CC \}'' \subset \Bd[L^2(\grp) \otimes \HH]$.
We write $\pi_{\hat{\alpha}} : \vNcrs \to \vN \btens \vNcrs $ for the natural coaction of $\grp$ on $\vNcrs$ (see \cite[Proposition 2.4]{NaTa79}), given by
\begin{gather*}
	\pi_{\hat{\alpha}} \big( \pi_\alpha(a) \big) \defeq 1 \otimes \pi_\alpha(a) ,  \quad  
	\pi_{\hat{\alpha}} (\lambda_r \otimes 1) \defeq \lambda_r \otimes \lambda_r \otimes 1 ,
\end{gather*}
for all $r \in \grp$ and all $a \in M$.
This coaction induces a module action of $\falg$ on $\vNcrs$, with the module structure given by
\[
	u * x \defeq (u \otimes \id) \pi_{\hat{\alpha}} (x) , \quad u \in \falg,\ x \in \vNcrs .
\]
In fact, with this definition $\vNcrs$ is a faithful $\falg$-module.  
This module structure induces an inclusion 
\[
	\Delta : \vNcrs \hookrightarrow \CB[\falg, \vNcrs] ;\ \Delta(x)(u) \defeq u * x , \quad u \in \falg,\ x \in \vNcrs .
\]
There is a natural $\falg$-module structure on $\CB[\falg, \vNcrs]$ given by
\[
	(u \cdot \Psi)(v) \defeq \Psi(v u) , \quad u,v \in \falg ,
\]
and $\Delta$ is an $\falg$-module map under this definition. 

\begin{remark}\label{re:inclusionDeltaandcoaction}
{\rm 
By \textit{e.g.}\ \cite[Corollary 7.1.5]{EfRu00} there is a natural isomorphism 
\[
	\CB[\falg, \vNcrs] \cong \vN \btens (\vNcrs) ,
\] 
and under this isomorphism the inclusion $\Delta$ is induced by the coaction $\pi_{\hat{\alpha}}$. Indeed, the image of $\Delta$ is identified by this isomorphism with the image of $\pi_{\hat{\alpha}}$ in $\vN \btens (\vNcrs)$.
}
\end{remark}

The coproduct on $\vN$ is implemented by the multiplicative unitary $U \in \vN \btens L^\infty(\grp)$, given by 
\[
    U \xi (s,t) \defeq \xi(t s , t) , \quad \xi \in L^2(\grp) \otimes L^2(\grp) ,\ s,t \in \grp ,
\]
so that $\coact[\vN](x) = U^* (1 \otimes x) U$ ($x \in \vN$). 
This extends to a coaction of $\grp$ on $\Bd[L^2(\grp)]$ by the same formula. 
Thus the module action $*$ of $\falg$ on $\vN$ extends to a module action on $\Bd[L^2(\grp)]$ by 
\[
	u * x \defeq (u \otimes \id) U^* (1 \otimes x) U , \quad u \in \falg,\ x \in \Bd[L^2(G)] .
\]

We include the following easy lemma because it is used several times in our arguments.

\begin{lemma}\label{le:equivariantmapextends}
Let $(M,G,\alpha)$ and $(N,G,\beta)$ be w*-dynamical systems and suppose that $\Phi : M \to N$ is a $\grp$-equivariant map, \textit{i.e.}\ $\Phi \circ \alpha_r = \beta_r \circ \Phi$ for all $r \in \grp$.
Then there is an $\falg$-module map $\tilde{\Phi} : \Bd[L^2(\grp)] \btens M \to \Bd[L^2(\grp)] \btens N$ which restricts to an $\falg$-module map $\tilde{\Phi} |_{\vNcrs} : \vNcrs \to \vNcros{N}{\grp}{\beta}$.
Moreover, if $\Phi$ is completely bounded then so is $\tilde{\Phi}$, and $\nrm[\tilde{\Phi}]_\cb \leq \nrm[\Phi]_\cb$.
\end{lemma}
\begin{proof}
Define $\tilde{\Phi} \defeq \id \otimes \Phi$; it is clear that $\tilde{\Phi}$ is an $\falg$-module map, and well-known that $\nrm[\tilde{\Phi}]_\cb \leq \nrm[\Phi]_\cb$.
Since $\Phi$ is $G$-equivariant we have, for $a \in M,\ r \in G,\ \xi \in L^2(G,\HH_N)$,
\[
	\tilde{\Phi}\big( \piact(a) \big) \xi(r) = \Phi \big( \alpha_{r\inv}(a) \big) \xi(r) = \beta_{r\inv}\big( \Phi(a) \big) \xi(r) = \piact[\beta]\big( \Phi(a) \big) \xi(r) .
\]
It follows that $\tilde{\Phi}(\piact(M)) \subset \piact[\beta](N)$, so the claim about the restriction of $\tilde{\Phi}$ follows since it acts as the identity on $\vN \otimes \CC \subset \vNcrs$.
\end{proof}

\subsection{Approximation properties for operator modules}
\label{ssec:moduleapproxprops}

The following definitions are given by Crann~\cite[Section 7]{Cra20}.

Let $X$ be a completely contractive Banach algebra. An operator space $A$ will be called a \emph{left operator module over $X$} if $A$ is a left $X$-module and the module action extends to a complete contraction $X \projtens A \to A$ (here $\projtens$ denotes the operator space projective tensor product).
If $\unitiz[X] \defeq X \oplus \CC$ denotes the unitisation of $X$ then the module action extends to a complete contraction $\unitiz[X] \projtens A \to A$ by $(x,c) \cdot a \defeq x \cdot a + ca$ ($x \in X,\ a \in A,\ c \in \CC$).
We will often omit the adjective ``operator'' in the rest of the paper, for example writing left module over $X$ in place of left operator module over $X$.
Write $X - \modu$ for the category of left $X$-modules with completely bounded module maps as morphisms.

\begin{definition}\label{de:nuclearmodmaps}
Let $X$ be a completely contractive Banach algebra and $A,B$ be left modules over $X$ (respectively, left modules over $X$ which are dual spaces).
A map $\theta : A \to B$ is called \emph{nuclear in $X - \modu$} (respectively \emph{weakly nuclear in $X - \modu$}) if there are morphisms: 
\[
	\varphi_k : A \to M_{n_k}(\unitiz[X]^*) , \quad \psi_k : M_{n_k}(\unitiz[X]^*) \to B ,
\]
such that $\psi_k \circ \varphi_k$ converges to $\theta$ in the point-norm (respectively, the point-weak*) topology.
\end{definition}

The following definitions are natural generalisations of the usual ones for operator spaces.

\begin{definition}\label{de:semidiscinjmodule}
Let $X$ be a completely contractive Banach algebra and $A$ be a left module over $X$. 
We say $A$ is \emph{nuclear in $X - \modu$} if the identity map $\id_A : A \to A$ is nuclear in $X - \modu$.
If $A$ is also a dual space then we say $A$ is \emph{semidiscrete in $X - \modu$} if the identity map $\id_A$ is weakly nuclear in $X - \modu$.
Finally, we say $A$ is \emph{injective in $X - \modu$} if, for any two $X$-modules $B$ and $C$, any morphism $\phi : B \to A$ and any completely isometric morphism $\kappa : B \to C$, there is a morphism $\tilde{\phi} : C \to A$ such that $\phi = \tilde{\phi} \circ \kappa$.
\end{definition}


%
%

We will make repeated use of the obvious fact that injectivity in $X - \modu$ is preserved under taking conditional expectations which are $X$-module maps.

The following results are known; we record them here for later use.

\begin{lemma}\label{le:injectivemodulesfacts}
\begin{enumerate}[(i)]
	\item For any locally compact group $\grp$ the algebra $\Bd[L^2(\grp)]$ is injective in $\falg - \modu$.
	\item Suppose that $M$ is an injective von~Neumann algebra. Then $\Bd[L^2(\grp)] \btens M$, endowed with the left $\falg$-module structure induced by that on $\Bd[L^2(\grp)]$, is injective in $\falg - \modu$. 
\end{enumerate}
\end{lemma}
\begin{proof}
(i) This is essentially proved in \cite[Theorem 5.5]{CrNe16}. 
Indeed, \cite{Ren72} shows, for any locally compact group $\grp$, that $\vN$ has an $\falg$-invariant state. 
Then the same proof as \cite[Theorem 5.5]{CrNe16} implies the claim.

(ii) If $M = \Bd[\HH]$ this follows routinely from (i) (see \textit{e.g.}\ \cite[Proposition XV.3.2]{Tak03c}). The general case follows.
\end{proof}

\subsection{Amenability and injectivity}
\label{ssec:injectivity}

Recall that a von~Neumann algebra $M$ is called \emph{injective} if for all unital $C^*$-algebras $A$ with a unital inclusion $M \subset A$ there is a projection of norm 1 from $A$ to $M$.
Equivalently, if $M \subset \Bd$ there is a projection of norm 1 from $\Bd$ to $M$. 
We say that a crossed product $\vNcrs$ is \emph{relatively injective in $\falg - \modu$} if there is a morphism $\Phi : \CB[\falg,\vNcrs] \to \vNcrs$ which is an $\falg$-module map and satisfies $\Phi \circ \Delta = \id$.
When $M = \CC$ and $\alpha$ is trivial this definition reduces to relative injectivity of $\vN$ as used by Crann--Tanko~\cite{CrTa17}.

Recall that a locally compact group is called \emph{amenable} if there is a state on $L^\infty(\grp)$ which is invariant under the left translation action $\tau$ of $\grp$ on $L^\infty(\grp)$.
A locally compact group is called \emph{inner amenable} if there is a state on $L^\infty(\grp)$ which is invariant under the conjugation action of $\beta$ of $\grp$ on $L^\infty(\grp)$. Crann--Tanko~\cite[Proposition 3.2]{CrTa17} showed that inner amenability of $\grp$ is equivalent to the existence of a state on $\vN$ which is invariant under the conjugation action of $\grp$ on $\vN$; it is this condition which we generalise in Definition~\ref{de:inneramenableactrelinj}. 
Note that there is another notion of inner amenable locally compact group, introduced by Effros~\cite{Eff75}, which is different to the one just introduced because it excludes the inner invariant state $\delta_e$. 

We will mostly use amenability in the form of injectivity properties of operator algebras. Here we summarise what is known about amenable groups in the language of injectivity. 
A discrete group is amenable if and only if its group von~Neumann algebra is injective (in $\CC - \modu$).


For locally compact groups inner amenability of $\grp$ is equivalent to relative injectivity of $\vN$ in $\falg - \modu$.


Crann--Tanko~\cite[Theorem 3.4]{CrTa17} observe that Lau--Paterson's result \cite[Corollary 3.2]{LaPa91} can be phrased in this way, generalising the above results.

\begin{theorem}\label{th:laupat}
Let $\grp$ be a locally compact group.
The following are equivalent:
\begin{enumerate}[(i)]
	\item $\grp$ is amenable;
	\item $\vN$ is injective in $\falg - \modu$; 
	\item $\vN$ is relatively injective in $\falg - \modu$ and injective in $\CC - \modu$.
\end{enumerate}
\end{theorem}

Our aim is to generalise this to crossed products.
Crann--Tanko~\cite[Theorem 3.4]{CrTa17} show how relative module injectivity is precisely what is needed for averaging arguments to work. 
This paper began with the idea to look for similar averaging techniques for crossed products.

\section{Amenable Actions on von Neumann Algebras}
\label{sec:amenablevNcase}

Recall Anantharaman-Delaroche's definition of an amenable action \cite{AD79}.  

\begin{definition}\label{de:ADamenableaction}
Let $\alpha$ be an action of a locally compact group $G$ on a von~Neumann algebra $M$.
We say that $\alpha$ is \emph{amenable} if there is a projection of norm $1$
\[
	P : L^\infty(G) \btens M \to M \ \text{such that $P \circ (\tau_r \otimes \alpha_r) = \alpha_r \circ P, \quad r \in G$.}
\]
That is, there is a $\grp$-equivariant conditional expectation from $L^\infty(\grp) \btens M$ to $M$.
Here $\tau$ is the natural action of $\grp$ on $L^\infty(\grp)$ by left translation.
\end{definition}

When $M = \CC$ and the action is trivial the projection $P$ is a left-invariant mean, so $G$ is amenable.

We now aim to improve \cite[Proposition 3.11]{AD79} by accounting for the $\falg$-module structure.
Recall that the coproduct on $\vN$ is unitarily implemented, so extends to $\Bd[L^2(\grp)]$ where it induces a module action of $\falg$.

\begin{proposition}\label{pr:ADprop311}
Let $\alpha$ be an action of $\grp$ on $M$.
The following are equivalent:
\begin{enumerate}[(i)]
	\item $\alpha$ is an amenable action;
	\item there is a norm 1 projection $\Bd[L^2(\grp)] \btens M \to \vNcrs$ which is an $\falg$-module map.
\end{enumerate}
\end{proposition}
\begin{proof} 
(i)$\implies$(ii) Let $P : L^\infty(\grp) \btens M \to M$ be the norm 1 equivariant projection from the definition of an amenable action.
By Lemma~\ref{le:equivariantmapextends} $P$ extends to an $\falg$-module map $\tilde{P} : \Bd[L^2(\grp)] \btens L^\infty(\grp) \btens M \to \Bd[L^2(\grp)] \btens M$, which restricts to a norm one $\falg$-module projection
\[
	P_\alpha : \vNcros{\big( L^\infty(G) \btens M \big)}{\grp}{\tau \otimes \alpha} \to \vNcrs .
\]
By \cite[Lemme 3.10]{AD79} $\vNcros{\big( L^\infty(\grp) \btens M \big)}{\grp}{\tau \otimes \alpha}$ is isomorphic to $\Bd[L^2(\grp)] \btens M$, so $P_\alpha$ is identified with a norm 1 projection $\Bd[L^2(\grp)] \btens M \to \vNcrs$.
Moreover, for each $r \in \grp$ the isomorphism in \cite[Lemme 3.10]{AD79} identifies $\lreg_r \otimes \id_{L^\infty(\grp) \btens M} \in \vNcros{\big( L^\infty(\grp) \btens M \big)}{\grp}{\tau \otimes \alpha}$ with $\lreg_r \otimes \id_M \in \Bd[L^2(\grp)] \btens M$, so it follows that the natural $\falg$-module structure on the domain of $P_\alpha$ is transformed under this isomorphism to the natural $\falg$-module structure on $\Bd[L^2(\grp)] \btens M$. 
It follows that the map which corresponds to $P_\alpha$ under this identification is an $\falg$-module map.

(ii)$\implies$(i) Let $P_\alpha$ be the projection in (ii) and take $x \in L^\infty(\grp) \btens M$ and $u \in \falg$; we have
\[
	u * P_\alpha(x) = P_\alpha(u * x) = u(e) P_\alpha(x) .
\]
As this holds for all $u \in \falg$ we conclude $P_\alpha(x) \in \piact(M) \cong M$. 
To see that the restriction of $P_\alpha$ to $L^\infty(\grp) \btens M$ is equivariant identify $P_\alpha$ with the corresponding map
\[
	\vNcros{(L^\infty(\grp) \btens M)}{\grp}{\tau \otimes \alpha} \to \vNcros{(\CC \btens M)}{\grp}{\id \otimes \alpha}
\] 
using \cite[Lemme 3.10]{AD79}. 
Then, for $x \in L^\infty(\grp) \btens M$, 
\[
	P_\alpha \big( \piact[\tau \otimes \alpha] \big( (\tau_r \otimes \alpha_r)x \big) \big) = P_\alpha \big( \lrega_r \piact[\tau \otimes \alpha](x) \lrega_r^* \big) = \lrega_r P_\alpha \big( \piact[\tau \otimes \alpha](x) \big) \lrega_r^*
\]
by the bimodule property of $P_\alpha$. 
Since $P_\alpha ( \piact[\tau \otimes \alpha](x) ) \in \piact[\id \otimes \alpha](M)$ this shows that $P_\alpha$ is equivariant.
\end{proof}

The $\falg$-module structure allows us to recover some results for locally compact groups which are only proved for discrete groups in \cite{AD79}.

\begin{remark}\label{re:AGamenactdefrmks}
{\rm 
Anantharaman-Delaroche \cite[Proposition 3.6]{AD79} has shown that amenable groups always have amenable actions. 
There are also many examples of non-amenable groups which have amenable actions; this follows from \textit{e.g.}\ \cite[Theorem 5.8]{BCL17}, since there exist non-amenable exact groups.
}
\end{remark}

The following result is given by Anantharaman-Delaroche~\cite[Proposition 3.6]{AD79}. We will use the implication (ii)$\implies$(i) later in Theorem~\ref{th:LauPatinneramenactions}, so we give a proof of this implication using our techniques. 
For a von~Neumann algebra $M$ we write $\centre$ for the centre of $M$.

\begin{proposition}\label{pr:AGamenactamengrp}
Let $\alpha$ be an action of $\grp$ on $M$.
The following are equivalent:
\begin{enumerate}[(i)]
	\item $\grp$ is amenable;
	\item $\alpha$ is amenable and $\centre$ has a $\grp$-invariant state.
\end{enumerate}
\end{proposition}
\begin{proof}
(i)$\implies$(ii) See Remark~\ref{re:AGamenactdefrmks} and \cite[Proposition 3.6]{AD79}.

(ii)$\implies$(i) Let $P_\alpha : \Bd[L^2(\grp)] \btens M \to \vNcrs$ be a norm one projection.
For each $x \in \Bd[L^2(\grp)]$ we claim $P_\alpha(x \otimes 1_M) \in \vNcros{\centre}{\grp}{\alpha}$.
We need only verify the claim for $x \in L^\infty(\grp)$; in this case the bimodule property of $P_\alpha$ means that for $y \in \piact(M)$
\[
	y P_\alpha(x \otimes 1_M) = P_\alpha \big( y (x \otimes 1_M \big) = P_\alpha \big( (x \otimes 1_M) y \big) = P_\alpha(x \otimes 1_M) y ,
\] 
so $P_\alpha(x) \in \vNcrs \cap \piact(M)' \subset \vNcros{\centre}{\grp}{\alpha}$. 
Now let $\phi : \centre \to \CC$ be a $\grp$-invariant state, and $\tilde{\phi} : \vNcros{\centre}{\grp}{\alpha} \to \vN$ the corresponding map given by Lemma~\ref{le:equivariantmapextends}. 
The composition $\tilde{\phi} \circ P_\alpha$ is a norm one $\falg$-module projection from $\Bd[L^2(\grp)] \otimes \CC$ to $\vN$, so $\grp$ is amenable by Theorem~\ref{th:laupat}.
\end{proof}

Now we show how accounting for the $\falg$-module structure helps the study of injectivity of crossed products; this result generalises \cite[Proposition 3.12, Corollaire 4.2]{AD79}.
Part of this result was also obtained recently by Crann~\cite[Theorem 8.2]{Cra20} and Bearden--Crann~\cite[Theorem 5.2]{BeCr20b}.

\begin{theorem}\label{th:injectivitynewamenability}
Let $\alpha$ be an action of $\grp$ on $M$.
The following are equivalent:
\begin{enumerate}[(i)]
	\item $\alpha$ is amenable and $M$ is injective;
	\item $\vNcrs$ is injective in $\falg - \modu$.
\end{enumerate}
\end{theorem}
\begin{proof}
(i)$\implies$(ii) Since $M$ is injective we have $\Bd[L^2(\grp)] \btens M$ is injective in $\falg - \modu$ by Lemma~\ref{le:injectivemodulesfacts}.
By Proposition~\ref{pr:ADprop311} there is a norm 1 $\falg$-module projection $P_\alpha : \Bd[L^2(\grp)] \btens M \to \vNcrs$.
Hence (ii) holds.

(ii)$\implies$(i) It is routine to obtain a norm 1 $\falg$-module projection $E : \Bd[L^2(\grp)] \btens \Bd \to \vNcrs$. 
Clearly $E$ restricts to a norm 1 projection $\Bd[L^2(\grp)] \btens M \to \vNcrs$ giving amenability of $\alpha$. 
Take $x \in L^\infty(\grp) \btens \Bd$ and $u \in \falg$, so that $u * x = u(e) x$. 
Since $E$ is an $\falg$-module map $u * E(x) = E(u * x) = u(e) E(x)$; as this holds for all $u \in \falg$ we conclude that $E(x) \in \piact(M)$.
Thus $E$ restricts to a norm one projection $L^\infty(\grp) \btens \Bd \to \piact(M) \cong M$, so $M$ is injective.
\end{proof}

\section{Inner Amenable Actions on von Neumann Algebras}
\label{sec:inneramenableactions}

Recall that a locally compact group $\grp$ is called inner amenable if there is an inner-invariant state on $L^\infty(\grp)$, and that Crann--Tanko~\cite[Proposition 3.2]{CrTa17} showed that inner amenability of $\grp$ is equivalent to the existence of an inner-invariant state on $\vN$, \textit{i.e.}\ a state $\phi$ on $\vN$ satisfying $\phi(\beta_r(x)) = \phi(x)$ ($r \in G , x \in \vN$), where
\[
	\beta_r(x) \defeq \lreg_r x \lreg_r^* , \quad x \in \vN ,\ r \in \grp .
\]
It is this latter condition which we generalise to define inner amenable actions.

\begin{definition}\label{de:inneramenableactrelinj}
Let $(M,\grp,\alpha)$ be a $w^*$-dynamical system. 
We say $\alpha$ is \emph{inner amenable} if there is a projection of norm 1 
\[
    Q : \vN \btens M \to M \text{ such that $Q \circ (\beta_r \otimes \alpha_r) = \alpha_r \circ Q, \quad  r \in \grp$.}
\]
That is, there is a $\grp$-equivariant conditional expectation from $\vN \btens M$ to $M$.
\end{definition}

If $M = \CC$ and $\alpha$ is trivial then the equivariant projection $Q$ in Definition~\ref{de:inneramenableactrelinj} is an inner-invariant state on $\vN$, so the above definition reduces to inner amenability of $\grp$ by \cite[Proposition 3.2]{CrTa17}.

\begin{remark}\label{re:inneramengrpsact}
{\rm 
To see that inner amenable groups have inner amenable actions suppose that $(M,G,\alpha)$ is a $w^*$-dynamical system, and that $G$ is inner amenable.
Let $\phi : \vN \to \CC$ be a state which is invariant under the action $\beta$ and define
\[
	Q : \vN \btens M \to M ;\ Q(x) \defeq (\phi \otimes \id) (x) , \quad x \in \vN \btens M .
\]
It is easily seen that $Q$ satisfies Definition~\ref{de:inneramenableactrelinj}, hence $\alpha$ is inner amenable.
}
\end{remark}

We would like to know if there is a way to formulate our definition of inner amenable actions in terms of an equivariant projection on $L^\infty(\grp) \btens M$, generalising the original definition of inner amenable groups as the existence of an inner-invariant mean on $L^\infty(\grp)$.

In order to use inner amenable actions for averaging arguments we now show that Definition~\ref{de:inneramenableactrelinj} is equivalent to a crossed product version of relative injectivity, which is exactly the property required for the arguments in Theorem~\ref{th:LauPatinneramenactions} and 
Section~\ref{sec:approximation}.

\begin{proposition}\label{pr:innamenactrelinjequiv}
Let $(M,\grp,\alpha)$ be a $w^*$-dynamical system.
The following are equivalent:
\begin{enumerate}[(i)]
	\item $\alpha$ is inner amenable;
	\item $\vNcrs$ is relatively injective in $\falg - \modu$ (\textit{i.e.}\ there is a norm 1 $\falg$-module projection from $\vN \btens (\vNcrs)$ onto $\piact[\hat{\alpha}](\vNcrs)$).
\end{enumerate}
\end{proposition}
\begin{proof}
(i)$\implies$(ii) Let $U \in \vN \btens L^\infty(\grp)$ be the unitary on $L^2(\grp) \otimes L^2(\grp)$ given by $U \xi(s,t) \defeq \xi(t s ,t)$ and let $V \defeq \sigma U$, where $\sigma$ is the flip operator.
Routine calculations show that for $x \in \vN$
\[
	V (x \otimes \id_{\vN}) V^* = \piact[\beta](x) \quad \text{and} \quad V^* (x \otimes \id_{\vN}) V = \coact[\vN](x) .
\]
Thus we identify the $\falg$-modules $\vNcros{\vN}{\grp}{\beta}$ with $\vN \btens \vN$ by conjugating with $\sigma V^* = \sigma U^* \sigma$ (the $\falg$-module structure on $\vN \btens \vN$ comes from applying the coproduct to the left component).
Under this identification $\vN \otimes \CC \subset \vNcros{\vN}{\grp}{\beta}$ is identified with $\coact[{\vN}](\vN)$.

Similarly, if $M$ acts on the Hilbert space $\HH$, conjugating by the unitary $\sigma V^* \otimes \id_\HH$ we identify $\vNcros{(\vN \btens M)}{\grp}{{\beta \otimes \alpha}}$ with $(\sigma \otimes \id)(\vN \btens (\vNcrs))(\sigma \otimes \id)$.
Under this identification $\vNcros{(\CC \otimes M)}{\grp}{\beta \otimes \alpha}$ is identified with $\piact[\hat{\alpha}](\vNcrs)$.

Now let $Q : \vN \btens M \to M$ be the norm one equivariant projection from Definition~\ref{de:inneramenableactrelinj}, so the $\falg$-module map $\tilde{Q}$ given by Lemma~\ref{le:equivariantmapextends} gives a norm one $\falg$-module projection
\[
	(V \otimes \id_\HH)^* \tilde{Q} (V \otimes \id_\HH) : \vN \btens (\vNcrs) \to \piact[\hat{\alpha}](\vNcrs) ,
\]
as required.

(ii)$\implies$(i) Let $Q_\alpha$ denote the projection in (ii), and as in the first part of the proof above conjugate with the unitary $(V \otimes \id_\HH)$ to identify $Q_\alpha$ with a norm one $\falg$-module projection from $\vNcros{(\vN \btens M)}{\grp}{\beta \otimes \alpha}$ onto $\vNcros{(\CC \otimes M)}{\grp}{\id \otimes \alpha} \cong \vNcrs$.
The rest of the proof proceeds as in Proposition~\ref{pr:ADprop311}.
\end{proof}

Since a left-invariant mean on $L^\infty(\grp)$ is automatically two-sided invariant, all amenable groups are automatically inner amenable.
We can generalise this fact to actions on injective von~Neumann algebras.

\begin{proposition}\label{pr:AGamenimpliesAGinneramen}
Suppose that $(M,\grp,\alpha)$ is a $w^*$-dynamical system, with $M$ injective and $\alpha$ amenable. 
Then $\alpha$ is inner amenable.
\end{proposition}
\begin{proof}
Since $\alpha$ is amenable and $M$ is injective it follows from Theorem~\ref{th:injectivitynewamenability} that $\vNcrs$ is injective in $\falg - \modu$. 
It follows easily that $\vNcrs$ is relatively injective in $\falg - \modu$, so $\alpha$ is inner amenable.
\end{proof}

%

\begin{examples}\label{ex:innamenactnotinnamengrp}
{\rm 
\begin{enumerate}[(i)]
	\item The action $\tau$ of $\grp$ on $L^\infty(\grp)$ is always amenable \cite[Section 1.4]{AD82}, hence always inner amenable. So $\Bd[L^2(G)]$ is relatively injective in $\falg - \modu$. 
	\item Let $\grp$ denote a second-countable, connected, non-amenable locally compact group, for example $\grp = \mathrm{SL}(2, \RR)$. Such groups cannot be inner amenable, but they are exact \cite[Theorem 6.8]{KW99}.
Brodzki--Cave--Li~\cite[Theorem 5.8]{BCL17} have shown that such $\grp$ admit an amenable action, say $\alpha$, on a compact space $X$, so it follows from Proposition~\ref{pr:AGamenimpliesAGinneramen} that $\alpha$ is inner amenable.
\end{enumerate}
}
\end{examples}

By analogy with the class of exact groups, it may be interesting to study the class of locally compact groups which admit an inner amenable action on $L^\infty(X)$ for some compact space $X$.
By Remark~\ref{re:inneramengrpsact} this class contains all inner amenable groups, and the same argument as in Example~\ref{ex:innamenactnotinnamengrp} part (ii) shows this class contains all second-countable exact groups. 

To close this section we generalise a result of Lau--Paterson~\cite[Theorem 3.1]{LaPa91} to actions on noncommutative von~Neumann algebras.

\begin{theorem}\label{th:LauPatinneramenactions}
Let $(M,\grp,\alpha)$ be a $w^*$-dynamical system. 
The following are equivalent:
\begin{enumerate}[(i)]
	\item $\grp$ is amenable and $M$ is injective;
	\item $\vNcrs$ is injective, $\alpha$ is inner amenable and $\centre$ has a $\grp$-invariant state.
\end{enumerate}	
\end{theorem}
\begin{proof}
(i)$\implies$(ii) Follows from Proposition~\ref{pr:AGamenactamengrp}, Theorem~\ref{th:injectivitynewamenability} and Remark~\ref{re:inneramengrpsact}.

(ii)$\implies$(i) Since $\alpha$ is inner amenable, by Proposition~\ref{pr:innamenactrelinjequiv}, we can upgrade injectivity of $\vNcrs$ to injectivity of $\vNcrs$ in $\falg - \modu$ as in \cite[Proposition 2.3]{Cra17}.
It follows that $M$ is injective as in the proof of Theorem~\ref{th:injectivitynewamenability}. 
To show $\grp$ is amenable we follow a similar idea to Proposition~\ref{pr:AGamenactamengrp}. 
Let $E : \Bd[L^2(\grp)] \btens \Bd[\HH_M] \to \vNcrs$ be a norm one $\falg$-module projection, and observe that for $x \in \Bd[L^2(\grp)]$ we have $E(x \otimes 1_M) \in \vNcros{\centre}{\grp}{\alpha}$, as in the proof of Proposition~\ref{pr:AGamenactamengrp}.
Let $\phi$ be a $\grp$-invariant state on $\centre$, and $\tilde{\phi} : \vNcros{\centre}{\grp}{\alpha} \to \vN$ the associated $\falg$-module map given by Lemma~\ref{le:equivariantmapextends}.
The composition $\tilde{\phi} \circ E$ is a norm one $\falg$-module projection from $\Bd[L^2(\grp)] \otimes \CC$ to $\vN$, so $\grp$ is amenable by Theorem~\ref{th:laupat}.
\end{proof}

\section{A Sample Averaging Argument on Crossed Products}
\label{sec:approximation}

Our definition of inner amenable actions is designed to enable averaging arguments.
One example of such an argument has already occurred in the proof that (ii) implies (i) in Theorem~\ref{th:LauPatinneramenactions}, where inner amenability of the action allows us to obtain injectivity of the crossed product in $\falg - \modu$ from the assumption that the crossed product is injective in $\CC - \modu$. 
In the setting of $C^*$-algebra crossed products discreteness of the acting group is used in the same way in \cite[Proposition 3.4]{MSTT18}.
In this section we briefly give a further example of this technique.

\begin{lemma}\label{le:averaging}
Let $(M,\grp,\alpha)$ be a $w^*$-dynamical system. 
Suppose that $\alpha$ is inner amenable, with associated norm one $\falg$-module projection $Q_\alpha : \vN \btens \vNcrs \to \pi_{\hat{\alpha}}(\vNcrs)$, and let $\Phi : \vNcrs \to \vNcrs$ be a completely bounded map.
Then 
\[
	S_\Phi : \vNcrs \to \vNcrs ;\ S_\Phi \defeq \piact[\hat{\alpha}]\inv \circ Q_\alpha \circ (\id \otimes \Phi) \circ \piact[\hat{\alpha}] 
\] 
is a completely bounded $\falg$-module map, with $\nrm[S_\Phi]_\cb \leq \nrm[\Phi]_\cb$.
\end{lemma}
\begin{proof}
Let $Q_\alpha : \vN \btens \vNcrs \to \pi_{\hat{\alpha}}(\vNcrs)$ be the $\falg$-module norm 1 projection given by inner amenability of $\alpha$ in Proposition~\ref{pr:innamenactrelinjequiv}.
It is easily checked that $\piact[\hat{\alpha}](u * x) = (u \otimes \id) * \piact[\hat{\alpha}](x)$ for all $x \in \vNcrs$. 
Then, for any $u \in \falg$ and $x \in \vNcrs$ we have
\[
\begin{split}
	S_\Phi(u*x) 
		&= \piact[\hat{\alpha}]\inv \circ Q_\alpha \circ (\id \otimes \Phi) \big( (u \otimes \id) * \piact[\hat{\alpha}](x) \big) \\
		&= \piact[\hat{\alpha}]\inv \circ Q_\alpha \circ \Big( (u \otimes \id) * \big( \id \otimes \Phi) \circ \piact[\hat{\alpha}](x) \big) \Big) \\
		&= u* \big(\pi_{\hat{\alpha}}\inv \circ Q_\alpha \circ (\id \otimes \Phi) \circ \pi_{\hat{\alpha}}(x)\big) 
		= u*S_\Phi(x) ,
\end{split}
\]
Hence $S_\Phi$ is an $\falg$-module map.	
The norm inequality is obvious. 
\end{proof}

A von~Neumann algebra $N$ is said to have the \emph{weak* completely bounded approximation property} (w*CBAP) if there exists a net of ultraweakly-continuous, finite-rank, completely bounded maps $(\Phi_i: N \to N)_i$ such that $\Phi_i \to \id_N$ in the point-ultraweak topology and a constant $C$ for which $\nrm[\Phi_i]_\cb \leq C$ for each $i$.
The Haagerup constant $\Lambda_\cb(N)$ is the infimum of those $C$ for which such a net $(\Phi_i)_i$ exists, and $\Lambda_\cb (N)= \infty$ if $N$ does not have the w*CBAP.

\begin{proposition}\label{pr:amenactCBAP}
Let $(M,\grp,\alpha)$ be a $w^*$-dynamical system. 
Consider the conditions:
\begin{enumerate}[(i)]
	\item $M$ has the weak* completely bounded approximation property;
	\item $\vNcrs$ has the weak* completely bounded approximation property.
\end{enumerate}
If $\alpha$ is amenable then (i) implies (ii).
If $\alpha$ is inner amenable then (ii) implies (i).
\end{proposition}
\begin{proof}
(i)$\implies$(ii) This was shown by Anantharaman-Delaroche~\cite[4.10]{AD85}.
	
(ii)$\implies$(i) Let $Q_\alpha : \vN \btens \vNcrs \to \piact[\hat{\alpha}](\vNcrs)$ be the norm 1 $\falg$-module projection arising from inner amenability of $\alpha$ in  Proposition~\ref{pr:innamenactrelinjequiv}. 
Let $(\Phi_i : \vNcrs \to \vNcrs)_i$ be a net of maps which implement the w*CBAP of $\vNcrs$, and define
\[
	S_{\Phi_i} : \vNcrs \to \vNcrs ;\ S_{\Phi_i} \defeq \piact[\hat{\alpha}]\inv \circ Q_\alpha \circ (\Phi_i \otimes \id) \circ \pi_{\hat{\alpha}}
\]
as in Lemma~\ref{le:averaging}.
The maps $(S_{\Phi_i})_i$ satisfy $\nrm[S_{\Phi_i}]_\cb \leq \nrm[\Phi_i]_\cb$,  also implement the w*CBAP of $\vNcrs$, and they are $\falg$-module maps.
Since $u * (S_{\Phi_i}(x)) = S_{\Phi_i}(u * x)$ for all $x \in \vNcrs$ and all $u \in \falg$ we have $S_{\Phi_i}(\piact(M)) \subset \piact(M)$.
It follows that $\piact(M) \cong M$ has the w*CBAP and $\Lambda_\cb(M) \leq \Lambda_\cb(\vNcrs)$.
\end{proof}	

The above proof illustrates one of our motivations for defining inner amenable actions: inner amenability of the action allows us to ``average'' maps which implement an approximation property of $\rcrs$ into $\falg$-module maps which implement the approximation property.
The resulting $\falg$-module maps may be viewed as Herz--Schur multipliers, as in \cite{McK19}; this is the perspective adopted in \cite{MSTT18}, where averaging arguments are explicitly used to produce Herz--Schur multipliers.

\section{Amenable actions on $C^*$-algebras}
\label{sec:amenactCalgs}

In this section $(A,G,\alpha)$ will be a $C^*$-dynamical system. 
Anantharaman-Delaroche~\cite[D\'{e}finition 4.1]{AD87} gave a definition of amenability for actions of discrete groups on $C^*$-algebras: if $(A,\grp,\alpha)$ is a $C^*$-dynamical system with $\grp$ discrete then $\alpha$ is called \emph{amenable} if the corresponding double dual action $\alpha^{**}$ of $\grp$ on $A^{**}$ is amenable in the sense of Definition~\ref{de:ADamenableaction}. 
In \cite[Section 4]{AD87} it is shown that this definition has several nice properties. 

The work of Anantharaman-Delaroche leaves open the question of defining amenable actions of locally compact groups on $C^*$-algebras.
It is proposed in \cite[Section 9.2]{AD02} that one might define an action of $\grp$ on $A$ to be \emph{amenable} if $\crs = \rcrs$, that is, if the canonical quotient map $\crs \to \rcrs$ is an isomorphism, and several questions about the behaviour of this proposed definition are raised.
Suzuki~\cite{Suz19} constructed examples showing that $\crs = \rcrs$ does not satisfy the functoriality properties in \cite[Section 9.2]{AD02}.
We interpret the questions asked by Anantharaman-Delaroche as requirements for a sensible definition of an amenable action.

\begin{problem}\label{pro:amenactCalgqus}
Let $(A,\grp,\alpha)$ be a $C^*$-dynamical system, with $\grp$ locally compact. 
Give a definition of \emph{amenability of $\alpha$} which satisfies the following properties:
\begin{enumerate}
	\item if $A$ is nuclear and $\alpha$ is amenable then the crossed product $\crs$ and/or the reduced crossed product $\rcrs$ is also nuclear;
	\item if $\alpha$ is amenable and $H$ is a closed subgroup of $\grp$ then the restriction of $\alpha$ is an amenable action of $H$ on $A$;
	\item if $(B,\grp,\beta)$ is a $C^*$-dynamical system such that there is an equivariant $*$-homomorphism $\Phi : A \to \Mult(B)$ with $\Phi(A) B$ dense in $B$ and $\Phi(\centre[\Mult(A)]) \subset \centre[\Mult(B)]$ then amenability of $\alpha$ implies amenability of $\beta$;
	\item if $A$ is a simple $C^*$-algebra then amenability of $\alpha$ implies amenability of $\grp$;
	\item if $\alpha$ is amenable then the canonical quotient map $\crs \to \rcrs$ is an isomorphism.
\end{enumerate}
\end{problem}

We have added condition (5) to the list in \cite[Section 9.2]{AD02} since it will not be part of the definition of an amenable action below.
Versions of this problem have been considered by a number of authors; we note the recent work of Bearden--Crann~\cite{BeCr20a} and Buss--Echterhoff--Willett~\cite{BEW20a}, which targets (5) above and along the way addresses a number of the other properties.  
Our main contribution is to the nuclearity problem (1): we use $\falg$-modules to relate amenable actions and nuclearity of crossed products.
We also survey known results, showing that the notion of amenable action we use satisfies several of the points in Problem~\ref{pro:amenactCalgqus}.

We follow Anantharaman-Delaroche~\cite[D\'{e}finition 4.1]{AD87} in defining an action of $\grp$ on a $C^*$-algebra $A$ to be amenable if a suitable double dual action is amenable in the sense of Definition~\ref{de:ADamenableaction}. 
The double dual action $\alpha^{**}$ on $A^{**}$ can fail to be weak*-continuous \cite{Iku88}, but fortunately there is a suitable replacement, which was introduced by Ikunishi~\cite{Iku88}. 

\begin{definition}\label{de:univenvelGvNalg}
Let $(A,\grp,\alpha)$ be a $C^*$-dynamical system.
Denote by $(i_A,i_\grp)$ the covariant representation underlying the universal representation $i_A \rtimes i_\grp : \crs \to \Bd[\HH_u]$.
We define $\Aact \defeq i_A(A)'' \subset \Bd[\HH_u]$.
\end{definition}

This definition is not the original one given by Ikunishi; we refer to Buss--Echterhoff--Willett~\cite[Section 2]{BEW20a}, where it is shown that Definition~\ref{de:univenvelGvNalg} is equivalent to Ikunishi's definition, and several useful properties of this object are shown.
The von~Neumann algebra $\Aact$ carries an action of $\grp$ given by $\Ad i_\grp$; we will abuse notation and write $\alpha^{**}$ for this action, which gives rise to a $w^*$-dynamical system $(\Aact , \grp , \alpha^{**})$.
To define an amenable action we follow \cite[Definition 3.1]{BEW20a}, where such actions are called \emph{von~Neumann amenable}.

\begin{definition}\label{de:amenactCalgebra}
Let $(A,G,\alpha)$ be a $C^*$-dynamical system.
Say that $\alpha$ is \emph{amenable} if the corresponding universal action $\alpha^{**}$ of $G$ on $\Aact$ is amenable in the sense of Definition~\ref{de:ADamenableaction}.
\end{definition}

For actions of discrete groups $\Aact = A^{**}$ \cite[Section 2.2]{BEW20a}, so this definition extends the one given by Anantharaman-Delaroche~\cite[D\'{e}finition 4.1]{AD87} for discrete groups. 
The remainder of this section is concerned with investigating if Definition~\ref{de:amenactCalgebra} has the properties in Problem~\ref{pro:amenactCalgqus}. 

First we consider nuclearity of crossed products, and aim to generalise \cite[Th\'{e}or\`{e}me 4.5]{AD87}. 
The assumption that the action is inner amenable is discussed in Remark~\ref{re:remonnuclthm}(iii) below.

\begin{theorem}\label{th:amenCactnuclear}
Let $(A,\grp,\alpha)$ be a $C^*$-dynamical system, and suppose that $\alpha^{**}$ is an inner amenable action of $\grp$ on $\Aact$.
The following are equivalent:
\begin{enumerate}[(i)]
	\item $\crs$ is nuclear;
	\item $\rcrs$ is nuclear;
	\item $\vNcros{\Aact}{\grp}{\alpha^{**}}$ is injective in $\falg - \modu$;
	\item $\alpha$ is amenable and $\Aact$ is injective. 
\end{enumerate}
\end{theorem}
\begin{proof}
(i)$\implies$(ii) Nuclearity passes to quotients.
	
(ii)$\implies$(iii) By hypothesis $(\rcrs)^{**}$ is injective (in $\CC - \modu$). Since $i_A$ is a faithful representation of $A$ we may form the covariant pair $((i_A)_\alpha , \lrega)$ as in equation (\ref{eq:covpair}). 
It is shown in \cite[Remark 2.6]{BEW20a} that 
\begin{equation}\label{eq:isomvNsecdualrcrs}
	\vNcros{\Aact}{\grp}{\alpha^{**}} = \big( (i_A)_\alpha \rtimes \lrega \big)^{**} \big( (\rcrs)^{**} \big),
\end{equation}
so it follows that $\vNcros{\Aact}{\grp}{\alpha^{**}}$ is injective (in $\CC - \modu$).
Since $\alpha$ is assumed to be inner amenable, $\vNcros{\Aact}{\grp}{\alpha^{**}}$ is relatively injective in $\falg - \modu$ by Proposition~\ref{pr:innamenactrelinjequiv}, so by a standard argument (see \cite[Proposition 2.3]{Cra17}) $\vNcros{\Aact}{\grp}{\alpha^{**}}$ is injective in $\falg - \modu$.

(iii)$\implies$(iv) This is shown in Theorem~\ref{th:injectivitynewamenability}.

(iv)$\implies$(i) Recent results allow us to adapt Anantharaman-Delaroche's proof \cite[Th\'{e}or\`{e}me 4.5]{AD87} to the locally compact case, as we now explain.
We will show that for a (non-degenerate) covariant representation $(\phi , \rho)$ of $(A,\grp,\alpha)$ on the Hilbert space $\HH$ the von~Neumann algebra $(\phi \rtimes \rho)(\crs)''$ is injective. In fact we will show the equivalent statement that the commutant $(\phi \rtimes \rho)(\crs)'$ is injective.
Observe (as in \cite[Th\'{e}or\`{e}me 4.5]{AD87}) that the latter is the space
\[
    \phi(A)'_\grp \defeq \{ x \in \phi(A)' : \rho_r x \rho_r^* = x \text{ for all $r \in \grp$} \} .
\]
Applying \cite[Corollary 2.3]{BEW20a} to the map $\phi : A \to \phi(A)$ we obtain a surjective, equivariant extension $\phi^{**} : \Aact \to \phi(A)''$. 
Since $\Aact$ is injective so is $\phi(A)''$, therefore also $\phi(A)'$.
Now, since $\alpha$ is amenable in our sense it is also amenable in the sense of \cite[Definition 3.4]{BEW20a}, by \cite[Theorem 3.6]{BeCr20a} (see also \cite[Section 8]{BEW20a}); therefore, by \cite[Proposition 5.9]{BEW20a} $\alpha$ is \emph{commutant amenable} in the sense of \cite[Definition 5.7]{BEW20a}.
Therefore, by \cite[Theorem 3.6]{BeCr20a}, the action $\Ad \rho$ of $\grp$ on $\phi(A)'$ given by 
\[
    \Ad \rho_r(x) \defeq \rho_r x \rho_r^* , \quad r \in \grp,\ x \in \phi(A)' ,
\]
is also amenable.
Let $R : L^\infty(\grp) \btens \phi(A)' \to \phi(A)'$ be the corresponding norm 1 equivariant projection and define
\[
    R_\grp : \phi(A)' \to \phi(A)'_\grp ;\ R_\grp(x) \defeq R(\hat{x}) , \quad x \in \phi(A)' ,
\]
where $\hat{x} \in  L^\infty(\grp) \btens \phi(A)'$ is defined by $\hat{x}(r) \defeq \Ad \rho_r(x)$.
Since $\phi(A)'$ is injective and $R_\grp$ is a norm 1 projection we have shown that $\phi(A)'_\grp$ is injective, as required.
%
%
\end{proof}	

\begin{remarks}\label{re:remonnuclthm}
{\rm 
\begin{enumerate}[(i)]
	\item The implication (iii)$\implies$(iv) fails if we do not account for the $\falg$-module structure: there exist locally compact groups, for example $\SL(2,\CC)$, for which the group von~Neumann algebra is injective but the group is not amenable \cite{Con76} (see also \cite[Remark 2.6.10]{BrO08}). Anantharaman-Delaroche works only with discrete groups, where this difficulty does not arise. See also Theorem~\ref{th:laupat}.
	\item Note that if $A$ is nuclear then $\Aact$ is injective, so condition (iv) above holds if $\alpha$ is amenable and $A$ is nuclear. If $\grp$ is discrete then $\Aact = A^{**}$, so (iv) is equivalent to \emph{(iv)' $\alpha$ is amenable and $A$ is nuclear}. In general we do not see any reason for (iv) and (iv)' to be equivalent.
	\item Our original goal was to prove a full generalisation of \cite[Th\'{e}or\`{e}me 4.5]{AD87}, without the assumption of inner amenable actions but using nuclearity of $\crs$ as an $\falg$-module in (i) (and nuclearity of $\rcrs$ as an $\falg$-module in (ii)). 
    Unfortunately it seems that such a result would require module versions of the deep results linking injectivity, semidiscreteness and nuclearity, which do not appear to be known in general. 
\end{enumerate}
}
\end{remarks}

The following shows that the class of inner amenable locally compact groups is an answer to \cite[Question 8.3]{BEW20a}.

\begin{corollary}\label{co:inneramengrpCalgamen}
Let $(A,\grp,\alpha)$ be a $C^*$-dynamical system.
Suppose that $\rcrs$ is nuclear and $\grp$ is inner amenable. 
Then $\alpha$ is amenable.
\end{corollary}
\begin{proof}
By hypothesis $(\rcrs)^{**}$ is injective, so it follows from \cite[Remark 2.6]{BEW20a} that $\vNcros{\Aact}{\grp}{\alpha^{**}}$ is injective.
Since $\grp$ is inner amenable it follows from Remark~\ref{re:inneramengrpsact} that $\alpha^{**}$ is inner amenable, that is, $\vNcros{\Aact}{\grp}{\alpha^{**}}$ is relatively injective in $\falg - \modu$.
Now \cite[Proposition 2.3]{Cra17} shows that $\vNcros{\Aact}{\grp}{\alpha^{**}}$ is injective in $\falg - \modu$.
By Theorem~\ref{th:injectivitynewamenability} $\alpha$ is amenable.
\end{proof}

For question (2) on restriction to closed subgroups we refer to Ozawa--Suzuki~\cite[Corollary 3.4]{OzSu20}, where the following result was recently shown.

\begin{proposition}\label{pr:restricttoclosedsubgrps}
Let $(A,\grp,\alpha)$ be a $C^*$-dynamical system with $\alpha$ amenable.
If $H$ is a closed subgroup of $\grp$ then the restriction $\alpha|_H$ is also amenable.
\end{proposition}

A solution to question (3) follows from work of Buss--Echterhoff--Willett and Bearden--Crann.

\begin{proposition}\label{pr:amenCactfunctorial}
Let $(A,\grp,\alpha)$ and $(B,\grp,\beta)$ be $C^*$-dynamical systems and $\Phi : A \to \Mult(B)$ a $\grp$-equivariant $*$-homomorphism such that $\Phi(A) B$ is dense in $B$ and $\Phi(\centre[\Mult(A)]) \subset \centre[\Mult(B)]$.
If $\alpha$ is amenable then so is $\beta$. 
\end{proposition}
\begin{proof}
By \cite[Theorem 3.6]{BeCr20a} $\alpha$ is amenable in our sense if and only if it is amenable in the sense of \cite[Definition 3.4]{BEW20a} (see also \cite[Section 8]{BEW20a}). 
Thus the claim is equivalent to \cite[Lemma 3.17]{BEW20a}.
\end{proof}

Examples given by Suzuki~\cite{Suz19} show that our definition of amenable action does not satisfy the requirements of problem (4) --- see \cite[Section 3]{BEW20b} and \cite[Corollary 3.13]{BEW20a}. 
However, we are able to solve problem (4) in the special case where $A$ is the compact operators on some Hilbert space (see \cite[Observation 5.24]{BEW20a}), as this is the only case for which we know that $\Aact$ is a factor.
It would be interesting to have other examples where $\Aact$ is a factor.

\begin{proposition}\label{pr:amenCactamengrp}
Let $(\Cpt,\grp,\alpha)$ be a $C^*$-dynamical system with $\alpha$ amenable and $\Cpt$ the $C^*$-algebra of compact operators on the Hilbert space $\HH$.
Then $\grp$ is amenable.
\end{proposition}
\begin{proof} 
Ikunishi~\cite[Example 1]{Iku88} shows that $\Cpt^{**}_\alpha = \Bd$, so the claim follows from Proposition~\ref{pr:AGamenactamengrp}.
\end{proof}

A solution to question (5) can also be deduced from the work of Bearden--Crann and Buss--Echterhoff--Willett.

\begin{proposition}\label{pr:amenCactregular}
Let $(A,\grp,\alpha)$ be a $C^*$-dynamical system with $\alpha$ amenable.
Then the canonical quotient map $\crs \to \rcrs$ is an isomorphism. 
\end{proposition}
\begin{proof}
It follows from the result of Bearden--Crann~\cite[Theorem 3.6]{BeCr20a} that if $\alpha$ is amenable in our sense then it is amenable in the sense of \cite[Definition 3.4]{BEW20a}. 
The conclusion then follows from \cite[Theorem 5.9]{BEW20a}.
\end{proof}

\begin{remark}\label{re:equivamenregimpossible}
{\rm 
Buss--Echterhoff--Willett~\cite[Proposition 5.28, Example 5.29]{BEW20a} have shown that the converse to Proposition~\ref{pr:amenCactregular} does not hold for general locally compact groups.  
We refer to \cite[Section 5]{BEW20a} for an investigation of other conditions which are related to this \emph{weak containment property}.
}
\end{remark}

\subsection*{Acknowledgements} 
The second author thanks Alireza Medghalchi for his continuous encouragement and the Department of Mathematics of Kharazmi University for support.
Parts of this work were completed when the second author was visiting the first author at Chalmers University of Technology and the University of Gothenburg; she would like to thank Department of Mathematical Sciences at Chalmers University of Technology and the University of Gothenburg for their warm hospitality. 
Further progress was made when the authors attended the 7th Workshop on Operator Algebras and their Applications in Tehran, in January 2020; we are very grateful to the organisers for their hospitality.
We thank Lyudmila Turowska and Massoud Amini for their valuable comments on early versions of our results and Fatemeh Khosravi for helpful discussions during this work.
Finally, thanks to Siegfried Echterhoff for pointing out a mistake in a previous version of Section~\ref{sec:amenactCalgs}.

\bibliographystyle{plain}
\bibliography{approxnondiscrete}

\end{document}